\documentclass[11pt]{article}
\usepackage{latexsym}
\usepackage{amssymb,amsmath,amsthm}
\usepackage{addfont} 
\addfont{OT1}{rsfs10}{\rsfs}
\usepackage{graphicx}
\RequirePackage{tikz}
\usepackage{epsfig}
\usepackage{multicol}
\usepackage{pst-grad} 
\usepackage{pst-plot} 
\usepackage{pst-node} 
\usepackage{pst-tree}
\usetikzlibrary{shapes, arrows}
\usepackage[all]{xy}
\graphicspath{{fig/}}
\usepackage{subfig}
\usepackage{amsfonts}

\theoremstyle{definition}
\newtheorem{theorem}{Theorem}[section]

\newtheorem{corollary}[theorem]{Corollary}
\newtheorem{proposition}[theorem]{Proposition} 
\newtheorem{problem}[theorem]{Problem}

\begin{document}

\title{{\bf Edge-Locating Coloring of Graphs}}

\author{M. Korivand$^{1}$, D.A. Mojdeh$^{2}$, Edy Tri  Baskoro$^{3,}$\thanks{Corresponding author} , and A. Erfanian$^{1}$ }

\date{}

\maketitle

\begin{center}

$^{1}$ Department of Pure Mathematics, Faculty of Mathematical Sciences and Center of Excellence in Analysis on
Algebraic Structures \\ Ferdowsi University of Mashhad, P.O. Box 1159-91775, Mashhad, Iran \\
e-mail: {\tt mekorivand@gmail.com }, {\tt erfanian@um.ac.ir}\\

$^{2}$ Department of Mathematics, Faculty of Mathematical Sciences \\
University of Mazandaran, Babolsar, Iran \\
 e-mail: {\tt damojdeh@umz.ac.ir}\\

$^{3}$ Combinatorial Mathematics Research Group, Institut Teknologi Bandung, and \\
Center for Research Collaboration on Graph Theory and Combinatorics, Indonesia. \\
e-mail: {\tt ebaskoro@itb.ac.id}

\end{center}

\begin{abstract}
An edge-locating coloring of a simple connected graph $G$ is a partition of its edge set into matchings such that the vertices of $G$ are distinguished by the distance to the matchings. The minimum number of the matchings of $G$ that admits an edge-locating coloring is the edge-locating chromatic number of $G$, and denoted by  $\chi'_L(G)$.
In this paper we initiate to introduce the  concept of edge-locating coloring and determine the exact values $\chi'_L(G)$ of some custom graphs. The graphs $G$ with  $\chi'_L(G)\in \{2,m\}$ are characterized, where $m$ is the size of $G$. We investigate the relationship between order, diameter, and  edge-locating chromatic number of $G$. For a complete graph $K_n$, we obtain the exact values of $\chi'_L(K_n)$ and $\chi'_L(K_n-M)$, where $M$ is a maximum matching;  indeed this result is also extended for any graph. We will determine the edge-locating chromatic number of join graph $G+H$, where $G$ and $H$ are some well-known graphs. In particular,
for any graph $G$, we show a relationship between $\chi'_L(G+K_1)$ and $\Delta(G)$.  We investigate the edge-locating chromatic number of trees and present a characterization bound for any tree in terms of maximum degree, number of leaves, and the support vertices of trees. Finally, we prove that any edge-locating coloring of a graph is an edge distinguishing coloring.

\end{abstract}


\medskip\noindent
{\bf AMS Subj.\ Class:} 05C15.

\section{Introduction}
One of the structural and applied topics in graph theory is distinguishing graph vertices and edges by means of different tools.
This approach has a relatively old history in graph theory and has used various tools such as distance and automorphism in graphs. In the following, we describe the history of some known concepts that follow such an approach. 

In 1977, Babaei proposed a concept that today inspires many methods for distinguishing elements of graphs by automorphism \cite{babai}. 
After Albertson and Collins \cite{alber} studied this concept in detail and proposed its application, it was widely considered in the name of {\it asymmetric coloring} (or {\it distinguishing labelling}). 
Among the parameters defined along this concept, we can mention {\it distinguishing coloring} (or {\it proper distinguishing coloring}),
{\it distinguishing index}, {\it distinguishing arc-coloring} and {\it distinguishing threshold} \cite{Collins, Kalinowski, Kalinowski2, Shekarriz}.

The other index related to automorphism is {\it determining set}, in which the goal is to identify the automorphism by a subset of graph vertices.
This concept were introduced independently by Boutin \cite{boutin} and Erwin \& Harary \cite{erwin}. 
The determining numbers of Kneser graphs and Cartesian product of graphs are provided in \cite{boutin, caceres, boutin2}.

One of the most important and well-known concepts that distinguishes the vertices of a graph with respect to distance is the {\it metric dimension}.
In 1975-76,  Slater \cite{slater} and Harary \& Melter \cite{harary}  independently introduced and studied this concept for connected graphs.
This introduction was a turning point for a branch of research that occupied many researchers, so that after about 50 years this concept is still the foundation of many research projects and applications, even in other sciences such as chemistry and computer science.
Due to its many applications in different sciences and other versions of the metric dimension, it has been introduced. In recent years, this concept has received more attention than in the past. We recommend the reader who needs more information about this concept refer to two recently raised surveys that discuss in detail the different versions of the metric dimension and its applications \cite{survey1, survey2}.

The {\it edge metric dimension} is one of these concepts derived from the metric dimension, where the goal is to distinguish the edges from a set of graph vertices \cite{Kelenc}.
Of course, in the metric dimension literature, we know the two concepts as edge metric dimension. The second case, which is also discussed in this article, means the least number of edges that resolve the vertices of a graph with respect to the distance \cite{nasir}.

In 2002, Chartrand et. al. introduced a coloring that we know as {\it locating coloring} \cite{chart}. In this coloring, the goal is to distinguish the vertices of a graph by their distance from a partition of the vertex set.
The locating coloring has been the subject of many researchers; for more details, see \cite{beh, chart03, Irawan, Mojdeh}.

In this paper, our goal is to distinguish the vertices of a connected graph by the distance of the matchings that partition the edge set. In fact, we can see this definition as the edge version of the locating coloring. We give its exact definition below.

Let $G$ be a simple connected graph. Let $c: E(G) \longrightarrow \mathbb{N}$
be a proper edge coloring of $G$, in which adjacent edges of $G$ have different colors. Let
$\pi = (\mathcal{C}_1, \mathcal{C}_2, \ldots, \mathcal{C}_k )$ denote the ordered partition of $E(G)$, that is the color classes admitted of $c$.
For a vertex $v$ of $G$, the {\em edge color code} $c_{\pi}(v)$ is the ordered $k$-tuple
$
({d}(v, \mathcal{C}_1), {d}(v, \mathcal{C}_2), \ldots, {d}(v, \mathcal{C}_k)),
$
where
${d}(v, \mathcal{C}_i)= \min \{{d}(v, e) | e \in \mathcal{C}_i \}$
for
$1\leq i \leq k$,
and
${d}(v, e)= \min \{{d}(v, x), {d}(v, y)  | e=xy\}$.

The coloring $c$ is called an {\it edge-locating coloring} of
$G$ if distinct vertices of $G$ have different edge color codes. The {\it edge-locating chromatic number} $\chi'_L(G)$ is the minimum number of colors needed for an edge-locating coloring of $G$.

In this paper, we generally seek to investigate the behavior of the edge-locating coloring in some family graphs.
Specifically, in Section 2, we compute the edge-locating coloring for paths, cycles, and complete bipartite graphs.
Also, we characterize all graphs $G$ of size $m$ with the property that
$\chi'_L(G) = k$, where $k\in \{2, m\}$.
Moreover, we present some bounds for the edge-locating chromatic number.
In Section 3, we  derive the edge-locating coloring of complete graphs and the complete graphs minus some matchings. Moreover,  in this section, we derive a sharp upper bound for the edge-locating chromatic number of a graph having a perfect matching and we extend it for a maximum matching.
In Section 4, we will determine the edge-locating chromatic number of join graph $G+H$, where $G$ and $H$ are some well known graphs.
In Section 5, we will examine the edge-locating chromatic number of trees. In particular, we compute the edge-locating chromatic number of the double star graphs and generalize it. Moreover, we present a characterization bound for any tree in terms of maximum degree, number of leaves and number of support vertices of trees.

We saw that there are several automorphism bases and distance bases coloring and index in graph theory. In general, these two concepts travel their research paths without paying attention to each other. However, some relationships between some of these parameters have been proven.
It has been shown that any resolving set of a graph is a determining set.
Determining sets and resolving sets were jointly studied in \cite{caceres2, garijo, pan}.
Also, Korivand, Erfanian, and Baskoro recently showed that any locating coloring is a distinguishing coloring \cite{korivand}.
In Section \ref{EdgeMetric-dist chro}, we prove that any edge-locating coloring of a graph is an edge distinguishing coloring. Also, we bound the edge-locating chromatic number to edge metric dimension and chromatic index.


\section{General results}

The edge-locating chromatic number is defined for graphs with more than two vertices. Since graphs are simple if all edges assign distinct colors then clearly the edge color codes of vertices are different. For any simple connected graph
$G$
with size
$m>2$,

$$
2 \leq \chi'_L(G) \leq m.
$$

Another natural bound for edge-locating chromatic number is
$\chi'(G)\leq \chi'_L(G)$.
Since
$\chi'(P_n) = 2$,
$\chi'_L(P_n) \geq 2$,
for
$n\geq 3$.
Clearly
$\chi'_L(P_3) =2$.
Assume that
$n > 3$.
If we consider an edge $2$-coloring of
$P_n$
then any two vertices of
$P_n$
that are not pendant vertices have the same edge color code. Thus
$\chi'_L(P_n) \geq 3$.
Now, for an edge-locating $3$-coloring of
$P_n$,
$n\geq 4$,
it is enough to assign color
$3$
to an edge with a pendant end vertex, and other edges of
$P_n$
coloring by color
$1$
and
$2$,
alternately. Therefore, 
$\chi'_L(P_n) = 3$.
Now, we can present the next proposition.

\begin{proposition}\label{prop 1}
For positive integer
$n$,
$\chi'_L(P_n)=\begin{cases}
2, & \text{if}\; n=3\\
3, & \text{if}\; n\geq 4.
\end{cases}$
\end{proposition}

The distance between two edges
$e_1$
and
$e_2$
is defined by
$\min \{{d}(a_i, b_j)\hspace{2mm} | \hspace{2mm} 1\leq i, j \leq 2, \hspace{2mm} e_1 = a_1 a_2, \hspace{2mm} e_2 = b_1 b_2 \}$.
\begin{theorem}\label{2}
For any integer
$n\ge 3$,
$\chi'_L(C_n)=\begin{cases}
3, & \text{if}\; n=3\\
4, & \text{if}\; n\geq 4.
\end{cases}$.
\end{theorem}

\begin{proof}
 For $n=3$, $\chi'(C_n)=\chi'_L(C_n)=3$. Now, we claim that
$\chi'_L(C_n) > 3$, 
for 
$n\geq 4$.
For a contradiction, assume that the edges of
$C_n$
are colored by three colors. Without loss of generality, we may suppose that the color
$3$
is the least used color in
$C_n$.
Let
$e=v_1v_2$
denote the only edge colored by
$3$.
Since
$n\geq 4$, for $n$ odd, the vertices $v_3, v_n$ have the same edge color code. For $n$ even, the vertices $v_1, v_2$ have the same edge color code. Hence, color
$3$
is assigned to at least two edges. Assume that
$e$
and
$f$
are two edges with color
$3$,
such that
${d}(e, f) = \min \{{d}(e_1, e_2) | e_1, e_2 \in \mathcal{C}_3 \}$.
If the distance between
$e$
and
$f$
are at least two, then
$c_{\pi}(a) = c_{\pi}(b)$,
where
$a\sim v\sim u\sim b$
and
$e=vu$.
Let
$\{ e_1, e_2, \ldots, e_m \}$
be a maximal alternative matching such that
${d}(e_i, e_{i+1}) = 1$
and
$e_i \in \mathcal{C}_3$,
for
$1\leq i \leq m$.
Since the color
$3$
is the least color used in
$C_n$,
the vertices
$a$
and
$b$
with the property that
${d}(a, e_1) = {d}(b, e_m) = 1$
and
$a$,
$b$
are not end points of
$e_i$,
$1\leq i \leq m$,
have the same edge color code
$(0, 0, 1)$.
Therefore, in all cases we have two vertices with the same edge color code, a contradiction.

Finally, we present an edge-locating $4$-coloring of
$C_n$
in such a way that assigns to two  incident edges colors
$3$
and
$4$,
and other edges coloring by
$1$
and
$2$, alternately.
\end{proof}
\begin{proposition}
Let $G$ be a graph. Then $\chi'_L(G)=2$ if and only if $G \cong P_3$.
\end{proposition}
\begin{proof}
Only one implication requires proof. Assume that $\chi'_L(G)=2$.
Hence $\Delta(G) =2$. This implies that $G$ is a cycle or a path. On the other hand, by Proposition \ref{prop 1} and Theorem \ref{2}, we know that all cycles and paths except
$P_3$ need at least three colors for an edge-locating coloring. So the result is immediate.
\end{proof}

\begin{theorem}\label{3}
For distinct integers $n, m\geq 2$, $\chi'_L(K_{n, m}) = \max \{n ,m\}+1$

and $\chi'_L(K_{n, n}) = n+2$.
\end{theorem}

\begin{proof}
For the comfort of calculation, we consider matrix
  $n \times m$, $A= [a_{b_i c_j}]$,
where
$\{b_1, b_2, \ldots, b_n \}$ and
$\{c_1, c_2, \ldots, c_m \}$ are the partite sets of $K_{n, m}$,
and $a_{b_i c_j}$ is the color of edge $b_i c_j$.
Thus, for any fixed integer
$i$ ($1\leq i \leq n$),
row
$(a_{b_i c_j})_{j=1}^{m}$ is the assigned colors of the incidence edges of
$b_i$.
Similarly, for any fixed integer
$j$ ($1\leq j \leq m$),
column
$(a_{b_i c_j})_{i=1}^{n}$
is the assigned colors of the incidence edges of
$c_j$.
An edge-locating coloring of
$K_{n, m}$
gives the following conditions on
$A$.
\begin{itemize}
\item[(i)] All elements in each row (column) are distinct.
\item[(ii)] For $i$ and $j$
($1\leq i, j \leq n$), $\{ a_{b_i c_k} \}_{k=1}^{m} \neq \{ a_{b_j c_k} \}_{k=1}^{m}$.
\item[(iii)] For $i$ and $j$ ($1\leq i, j \leq m$), $\{ a_{b_k c_i} \}_{k=1}^{n} \neq \{ a_{b_k c_j} \}_{k=1}^{n}$.
\end{itemize}
Let $n > m$.
To satisfy conditions (i) and (ii) we need more than $n$
colors. We claim that with $n+1$ colors matrix $A$
with conditions (i), (ii) and (iii) is constructed. For this, let
$S$ be $(n+1) \times (n+1)$ matrix consisting of all column matrices of colors
$ [1, 2, \ldots, n, n+1]^{t}$,
$[n+1, 1, \ldots, n-1, n]^{t},\ldots, [2, 3, \ldots, n+1, 1]^{t}$.
Now, assume that $A$ is the sub-matrix of $S$ consisting of first
$n$ rows and $m$ columns, where   $(a_{b_i c_1})_{i=1}^{n} = [1, 2, \ldots, n]^{t}$,
$(a_{b_i c_2})_{i=1}^{n} = [n+1, 1, \ldots, n-1]^{t},\ldots, (a_{b_i c_m})_{i=1}^{n} = [n+3-m, n+4-m, \ldots, n+m+1-m, 1,2, \ldots, n+1-m ]^{t}$.
Then, we can see that all conditions (i), (ii), and (iii) satisfy on
$A$, and the result is available.

For the other implication, let $m=n$. In this case, to construct matrix
$A$, we need condition (iv) in addition to conditions (i), (ii), and (iii).
\begin{itemize}
\item[(iv)]
For $i$ and $j$ ($1\leq i, j \leq n$), $\{ a_{b_i c_k} \}_{k=1}^{n} \neq \{ a_{b_k c_j} \}_{k=1}^{n}$.
\end{itemize}
According to the additional condition (iv), all $2n$ vertices should have distinct edge color sets that are incident to them. First, show that we cannot  edge-locating color of $K_{n, n}$ with $n+1$ colors. On the contradiction, assume that we can do it. Let the first column
be colored with $n$ colors, and don't use color $n+1$. Hence, all rows use $n+1$ as a color for a vertex. This shows that every row does not have one of the colors $i$ for $1\le i\le n$. On the other hand, since the first column does not use  color $n+1$, at least one column has $n+1$ as a color and does not take one of the colors in $1\le i \le n$ say $j$ and thus, this column and the row, which has no $j$ as  a color, have the same edge color code. That is a contradiction. Therefore, 
$\chi'_L(K_{n, n})\ge n+2$.

In the following, we give a way of edge-locating coloring
$K_{n, n}$,
by presentation  $n \times n$ matrix

\begin{center}
$A= \begin{pmatrix}
1 & 2 & 3 & \ldots & n-1 & n \cr
2 & 3 & 4 & \ldots & n & n+1 \cr
3 & 4 & 5 & \ldots & n+1 & n+2 \cr
4 & 5 & 6 & \ldots & n+2 & 1 \cr
\vdots & \vdots & \vdots &  & \vdots & \vdots \cr
n-1 & n & n+1 & \ldots & n-5 & n-4\cr
n+1 & n+2 & 1 & \ldots & n-3 & n-2
\end{pmatrix}. $
\end{center}
All colors are on module
$n$.
Also,
$a_{b_n, c_i} = a_{b_{n-1}, c_i} + 2$,
for
$1\leq i \leq n$.
One can check that matrix $A$ satisfies conditions (i) - (iv).
\end{proof}

In Figure 1, we give an illustration of Theorem \ref{3}, when
$n= 3$
and
$m=2$.
In this case, the edge-locating chromatic number is $4$, and the matrix
$A$
is
$ \begin{pmatrix}
1 & 4\cr
2 & 1\cr
3 & 2
\end{pmatrix} $.

\begin{figure}[!h]
\begin{center}
\label{FigE2}
  \includegraphics[width=8cm]{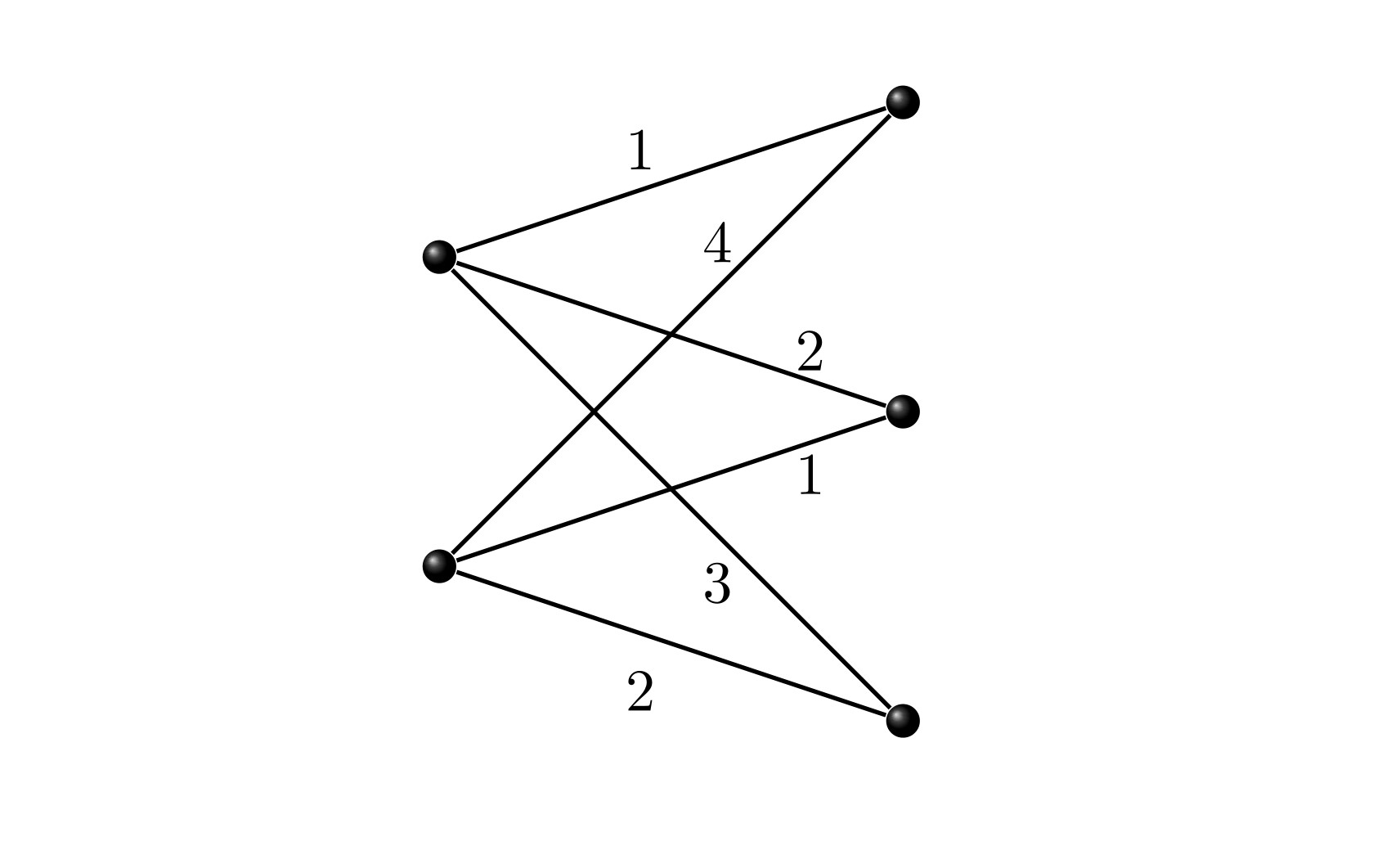}
\end{center}
\caption{\ An edge-locating coloring for $K_{3, 2}$.}
\end{figure}

For an integer
$n$,
the graph
$K_{1, n}$
is called a star graph and is shown by
$S_n$.
\begin{theorem}
Let
$G$
be a graph with size
$m\geq 2$.
Then
$\chi'_L(G) = m$
if and only if
$G \in \{P_4, C_3, C_4, S_m \}$.
\end{theorem}
\begin{proof}
If
$G \in \{P_4, C_3, C_4, S_m \}$,
then we have nothing to prove. For the other, assume first that
$\Delta(G) = 2$.
So, Proposition \ref{prop 1} and Theorem \ref{2} conclude that
$G \in \{P_4, C_3, C_4, S_2\}$.
Let
$\Delta(G) = k$,
for
$k\geq 3$.
For a contradiction, suppose that
$G \ncong S_m$,
for any
$m \geq 3$.
Let $v$ be a vertex of $G$ with ${deg}(v)=k\le m-1$.
Thus, there exists at least a vertex $u \neq v$ of $G$ such that
$1 < {deg}(u) \leq k$. Hence, there exists edge $e=uw$ in $G$, where $w \neq v$.
Since $k \geq 3$, we have edges $e'=vz$ and $e''=vx$ such that at least one of $z$ or $x$  is not in $\{u, w\}$.
Now, we assign color
$1$
to edges
$e$
and
$e'$,
and color the other edges with distinct colors
$2, 3, \ldots, m-1$ such that, without loss of generality, it is assigned color $2$ to edge $vu$ and color $3$ to edge $e''=vx$.
We will show that this coloring is an edge-locating coloring of
$G$.
For this, we have $c_{\pi}(v)=(0,0,0,\ldots)$,   $c_{\pi}(u)=(0,0,1,\ldots)$, $c_{\pi}(z)=(0,1,1,\ldots)$,  $c_{\pi}(w)=(0,1,2,\ldots)$ if $w$ is not adjacent to $x$, and if $w$ is  adjacent to $x$ the color of $wx$ is $4$, then $d(w,C_4)=0$  and $d(z,C_4)\le 1$. Therefore, these five vertices
have distinct edge color codes. For a vertex $y \notin \{v,u,w,z,x\}$, it is incident to at least one new edge with a new color. Hence $c_{\pi}(y)
\ne c_{\pi}(t)$ for $t\ne y$. This is a contradiction, and then $G=S_m$.
\end{proof}

In the following, we present some bounds for edge-locating coloring of a graph.
\begin{theorem}
Let
$G$
be a graph with order
$n$
and
$diam(G)=d\geq 3$. Then,

$$
\log_{d}n + 2 \leq \chi'_L(G).
$$
\end{theorem}
\begin{proof}
The edge color code of any vertex of
$G$
has
$\chi'_L(G)$
coordinates. Since each vertex is incident to at least one edge, at least one coordinate is $0$.
Let
$v$
be a vertex, and
$e=vu$.
There exists an edge
$e'=uw$
that
$w \neq v$.
The color of
$e'$
is different from
$e$, 
and the coordinate of the edge color code of
$v$
according to color
$e'$
is
$1$.
So, the two coordinates of any vertex of
$G$
are determined, and other coordinates can be filled by
$k$, $0\leq k \leq d-1$.
Since in any edge-locating coloring, each vertex must have a unique edge color code,
$n\leq d^{(\chi'_L(G)-2)}$,
and the result is obtained.
\end{proof}

\begin{theorem}
Let
$G$
be a graph with
$diam(G)=d\geq 3$
and
$\chi'_L(G)= k$. Then,

$$
\log_{(d-1)} [\dfrac{n_i}{\binom{k}{i}}] + i \leq k, \hspace{3mm} for \hspace{2mm} 1\leq i \leq \Delta.
$$

Where,
$n_i$
is the number of vertices of degree
$i$.
\end{theorem}
\begin{proof}
We can color the incident edges of a vertex of
$G$
of degree
$i$
with
$\binom{k}{i}$
ways. Thus,
$[\dfrac{n_i}{\binom{k}{i}}]$
numbers of vertices of degree
$i$
have the same colored incident edges. So, the other coordinates of this vertices can be filled by
$\ell$, $1\leq \ell \leq d-1$.
Therefore,
$[\dfrac{n_i}{\binom{k}{i}}]\leq (d-1)^{(\chi'_L(G)-i)}$,
for any
$1\leq i \leq \Delta$,
and the result is immediate.
\end{proof}

\section{Complete graphs and matchings}
In this section, we determine the edge-locating coloring of complete graphs and the complete graphs minus some matchings. Then we generalize this subject to  arbitrary graphs.
\subsection{Complete graphs}
A {\em matching} $M$ of a graph is a set of independent edges. A vertex is $M$-{\em saturated}
if it is incident with an edge of $M$, and $M$-{\em unsaturated} otherwise. A matching is said to be {\em maximum} if for any other matching $M^*, |M| \geq |M^*|$. A matching $M$ is perfect if it saturates all vertices of $G$.
Let
$K_n - e$
denote the complete graph
$K_n$
minus one edge.
\begin{theorem}\label{evencomplete}
For any even $n \geq 4$, $\chi'_L(K_n) = n+1$.
\end{theorem}

\begin{proof}
First, we show that $\chi'_L(K_n) \ge n+1$. Let $n$ be an even integer with $n \geq 4$.
Let $V(K_n)=\{v_1, v_2, \ldots, v_n\}$. 
Then $\chi'_L(K_n) \geq n$.
However, we will show that $\chi_L(K_n) \not= n$ for any even $n \geq 4$.
Let $c$ be any proper edge  locating  coloring of $K_n$ with $n$ colors. Then, each of at least $n/2$ colors will appear exactly
$n/2$ times each, and each of at most $n/2$ colors will appear at most
$n/2-1$ times each. A simple verification shows that,
precisely, $n/2$ different colors (say, colors $1, 2, \dots, \frac{n}{2}$) appear $n/2$ times each, and
other colors (namely, colors $\frac{n}{2}+1, \frac{n}{2}+2, \dots, n$) will appear exactly $n/2 -1$ times each.
Therefore, every vertex is incident to all colors except color $k$ for some $k \in \{\frac{n}{2}+1, \frac{n}{2}+2, \cdots, n\}$.
This means that there are only $\frac{n}{2}$ different edge color codes  for all $n$ vertices of $K_n$
with respect to coloring $c$. Thus, $c$ is not an edge-locating coloring of $K_n$,
and so $\chi'_L(K_n) \geq n+1$ for any even $n \geq 4$.\\

Now we provide an edge-locating coloring of $K_n$ with $n+1$ colors. As it is well known, the edge color code of any vertex $v$ is formed by $n+1$
coordinates, in which two of its coordinates are $1$ and the others are $0$.
Let $e_{ij}$ be an edge of $K_n$ with two end vertices $v_i$, and $v_j$ where $i<j$.\\
For defining $n+1$-edge-locating coloring function $\alpha$ on $K_n$, we consider two cases.\\
If $3\nmid n+1$, then we define $\alpha$ on $K_n$ as follows. $$\alpha(e_{ij})=j+i-2\ (\emph{mod}\ n+1)\ \emph{for}\ 1\le i<j\le n.$$
In this case, for any vertex $v_i$ two coordinates $2i-2$ and $i-2$ of the edge color code of $v_i$ are $1$ and the others are $0$.
Since $2i-2=j-2$ and $i-2=2j-2$, $3i=0 (\emph{mod} n+1)$. If $i\ne j$, then $\{2i-2,i-2\}\ne \{2j-2,j-2\}$.
If $3\mid n+1$ and $n+1=3k$, then we define $\alpha$ on $K_n$ as follows.\\ If  $ e_{ij}\notin\{e_{(lk-1)}, e_{(lk)}: 1\le l\le k-2\}$
$$\alpha(e_{ij})=j+i-2\ (\emph{mod}\ n+1)\ \emph{for}\ 1\le i<j\le n.$$  For
$ e_{ij}\in\{e_{(l(k-1))}, e_{(lk)}: 1\le l\le k-2\}$, we define
$$\alpha(e_{l(k-1)})=k+l-2\ (l<k-2);\ \alpha(e_{lk})=k+l-3\ (\emph{mod}\ n+1).$$
In this case, for any vertex $v_i, (i\ne k-1)$, the two coordinates $2i-2$ and $i-2$ of the edge color code of $v_i$ are $1$ and the others are $0$.
For $v_{k-1}$ the two coordinates $k-2$ and $k-3$ of the edge color code of $v_{k-1}$ are $1$ and the others are $0$. Similar to the above method,
one can show that, for $k-1\notin \{i, j\}$,  $\{2i-2,i-2\}\ne \{2j-2,j-2\}$ and for $j\ne k-1$,  $\{2j-2,j-2\}\ne \{k-2,k-3\}$.
\end{proof}

For instance, consider the edge-locating colorings of $K_{8}$  and $K_{10}$ represented by the two matrices $(8 \times 8)$-matrices and $(10 \times 10)$-matrices below, where $8+1=9=3\times 3$ and $3\nmid 10+1$.
The entries $ij$ are the colors of the edge $e_{ij}$ with two end-vertices $v_i$ and $v_j$, where $i<j$. In the matrix, we only wrote the color of $e_{ij}$ with $i<j$. For instance, in $K_8$, the vertex $v_3$ is incident to the colors\\ $\{c(e_{13})=3,\ c(e_{23})=4,\ c(e_{34})=5,\ c(e_{35})=6,\ c(e_{36})=7,\ c(e_{37})=8,\ c(e_{38})=9\}$\\ or in $K_{10}$, the vertex $v_4$ is incident to the colors\\ $\{c(e_{14})=3,\ c(e_{24})=4,\ c(e_{34})=5,\ c(e_{45})=7,\ c(e_{46})=8,\ c(e_{47})=9,\ c(e_{48})=10,\ c(e_{49})=11,\ c(e_{4\ 10})=1\}$.

\begin{center}
$
K_{8} :
\begin{pmatrix}
   - & 1 & 3 & 2 & 4 & 5 & 6 & 7\\
   - & - & 4 & 3 & 5 & 6 & 7 & 8\\
   - & - & - & 5 & 6 & 7 & 8 & 9\\
   - & - & - & - & 7 & 8 & 9 & 1\\
   - & - & - & - & - & 9 & 1 & 2\\
   - & - & - & - & - & - & 2 & 3\\
  - &  - & - & - & - & - & - & 4\\
   - &  - & - & - & - & - & - & - \\

\end{pmatrix}
3\mid 8+1=9
$

\vspace{3mm}

$K_{10} :
\begin{pmatrix}
     - & 1 & 2 & 3 & 4 & 5 & 6 & 7 & 8 & 9\\
     - &  - & 3 & 4 & 5 & 6 & 7 & 8 & 9 & 10\\
      - & - & - & 5 & 6 & 7 & 8 & 9 & 10 & 11\\
     - &  - & - & - & 7 & 8 & 9 & 10 & 11 & 1\\
      - & - & - & - & - & 9 & 10 & 11 & 1 & 2\\
     - &  - & - & - & - & - & 11 & 1 & 2 & 3\\
     - & - & - & - & - & - & - & 2 & 3 & 4\\
     - & - & - & - & - & - & - & - & 4 & 5 \\
     - & - & - & -  & - & - & - & - & - & 6\\
      - & - & - & -  & - & - & - & - & - & -\\

\end{pmatrix}
    3\nmid 10+1=11
$

\end{center}

\vspace{3mm}

\begin{theorem}\label{complete}
For any odd $n \geq 3$, $\chi'_L(K_n)=n$.
\end{theorem}

\begin{proof}
Let $n$ be an odd integer with $n \geq 3$.
Let $V(K_n)=\{v_1, v_2, \cdots, v_n\}$. Since $\Delta(K_n)=n-1$ then $\chi'_L(K_n) \geq n$.
We are going to  show that $\chi_L(K_n)=n$ for any odd $n \geq 3$.
We define $\alpha$ on $K_n$ as follows: $$\alpha(e_{ij})=j+i-2\ (\emph{mod}\ n)\ \emph{for}\ 1\le i<j\le n.$$
In this case, for any vertex $v_i$, one coordinate $(2i-2)$  of the edge color code of $v_i$ is $1$ and the others are $0$.
Since for $i\ne j$, $2i-2\ne 2j-2$, this coloring is an edge-locating coloring. Thus $\alpha$ is an edge-locating chromatic coloring of $K_n$, and so $\chi'_L(K_n) = n$ for odd $n$.
\end{proof}

\begin{center}

$K_{11}:
\begin{pmatrix}
      -\ \ \ 1 & 2 & 3 & 4 & 5 & 6 & 7 & 8 & 9 & 10\\
       -\ \ \ - & 3 & 4 & 5 & 6 & 7 & 8 & 9 & 10 & 11\\
       -\ \ \ - & - & 5 & 6 & 7 & 8 & 9 & 10 & 11 & 1\\
       -\ \ \ - & - & - & 7 & 8 & 9 & 10 & 11 & 1 & 2\\
      -\ \ \ - & - & - & - & 9 & 10 & 11 & 1 & 2 & 3\\
      -\ \ \ - & - & - & - & - & 11 & 1 & 2 & 3 & 4\\
      -\ \ \ - & - & - & - & - & - & 2 & 3 & 4 & 5\\
      -\ \ \ - & - & - & - & - & - & - & 4 & 5 & 6\\
      -\ \ \ - & - & - & - & - & - & - & - & 6 & 7\\
      -\ \ \ - & - & - & - & - & - & - & - & - & 8\\
      -\ \ \ - & - & - & - & - & - & - & - & - & -\\
\end{pmatrix}
$

\end{center}

For $k \geq 1$, let $M_k$ be a matching with $k$ edges.\\
In the proof of Theorem below, we can also use the method of the proof of Theorem \ref{KnminusMatchingeven} by the proof of Theorem \ref{complete}.
\begin{theorem}\label{KnminusMatchingodd}
For $n \geq 1$ and $1 \leq k \leq n$, we have that
\[  \chi_L'(K_{2n+1}\backslash M_k) = \left\{
\begin{array}{ll}
      2n+1  &  \mbox{, if $1 \leq k \leq n-1$,}   \\
      2n & \mbox{, if $k=n$.} \\
\end{array}
\right. \]
\end{theorem}
\begin{proof}
For $1 \leq k \leq n-1$, the graph $K_{2n+1}\backslash M_k$ has at least two vertices of degree $2n$.
Thus, $\chi_L'(K_{2n+1}\backslash M_k) \geq 2n+1$. But if $k=n$, since $K_{2n+1}\setminus M_n$ has exactly one vertex of degree $2n$, then $\chi_L'(K_{2n+1}\backslash M_n) \geq 2n$.
To obtain an edge-locating $(2n)$-coloring of
$K_{2n+1}\backslash M_n$ from the edge-locating $(2n+1)$-coloring of
$K_{2n+1}$ (in Theorem \ref{complete}) we can  remove the edges of a monochromatic $M_n$. Thus, the proof is observed. 
\end{proof}

If we note the proof of Theorem \ref{evencomplete}, there exists a $\chi_L'(K_{2n})$ with $2n+1$ colors such that each edge-locating
color class has at least $n-1$ edges and we can easily to see exactly $2n$ color classes have $n-1$ edges, and one color class has $n$ edges. The edge coloring is such that it can be said that color $2n-1$ was used for $n$ edges, and the rest of the colors were used for $n-1$ edges each. Therefore, we have.
\begin{theorem}\label{KnminusMatchingeven}
Let  $n \geq 2$. Then 
  $\chi_L'(K_{2n}\backslash M_{k})=2n$ if $k \in \{n-1, n\}$.
\end{theorem}
\begin{proof}
By Theorem \ref{evencomplete}, we have that $\chi_L'(K_{2n})=2n+1$ for $n \geq 2$.
 As we mentioned in the above, there is
exactly one perfect matching $M_n$ with  a monochromatic and $2n$ matchings $M_{n-1}$ with  a monochromatic each.
Thus, we get an edge-locating $2n$-coloring of $K_n \backslash M_k$ for $ k\in \{n-1, n\}$.
\end{proof}

As an immediate result of Theorems \ref{evencomplete} and \ref{KnminusMatchingeven}, we have.
\begin{corollary}
Let $m$ be a positive integer and $m\le n-1$. Then there exist $m$ matchings $M_{n-1}$ in which $\chi_L'(K_{2n}-m(M_{n-1})\cup M_n)=2n-m$.

 \end{corollary}
\subsection{Matchings}
In other words, an edge-locating coloring of a graph $G$
is a partition of its edge set into matchings such that the vertices of
$G$ are distinguished by the distance of the matchings. The minimum number of the matchings of
$G$ that admit an edge-locating coloring is the edge-locating chromatic number of
$G$.
\begin{theorem}\label{the-match}
Let $G$ be a graph with order $n\geq 5$ and size $m$.
If $G$ has a perfect matching, then $\chi'_{L}(G) \leq m- n/2 + 1$.
This bound is sharp for cycle $C_6$ and path $P_6$.
\end{theorem}
\begin{proof}
Let
$M$
be a perfect matching of
$G$.
Color all edges of
$M$
with color
$1$,
and other edges with distinct colors. We will show that this coloring is an edge-locating coloring of
$G$.
Note that vertices of
$G$
cannot be distinguished by the color
$1$.
Consider an arbitrary vertex
$v$
of
$G$.

Suppose first that
${N}(v) = \{u\}$.
So,
$vu \in M$.
It is enough to investigate the vertices that have distance one from edge(s)
$e=uw$,
for
$w \in {N}(u) \setminus \{v\}$.
Let
${N}(u) \setminus \{v\}=\{w\}$.
If
$\deg(w) \geq 3$,
there exists a vertex
$x$
adjacent to
$w$
such that
$xw \notin M$.
Hence, any vertex of
${N}(w)\setminus \{u\}$
and vertex
$v$
are distinguished by the color of
$xw$.
If
$\deg(w)=2$,
since
$n\geq5$,
there exists an edge
$f=zy$
such that
$\{z\}= {N}(w)\setminus \{u\}$
and
$y \notin \{v, u, w\}$.
Thus,
$f \notin M$,
and
$v$
and
$z$
are distinguished by the color of
$f$.
Assume that
$|{N}(u) \setminus \{v\}| \geq 2$.
A vertex
$z$
has distance one from edges
$e=uw$,
for
$w \in {N}(u) \setminus \{v\}$,
when
${N}(u) \setminus \{v\} \subseteq {N}(z)$.
In this situation, there exists a vertex
$x \in {N}(u) \setminus \{v\}$
such that
$xz \notin M$.
Therefore,
$v$
and
$z$
have different distances from
$xz$,
and the result is immediate.

Assume that
$\deg(v)=2$.
There exists at least an edge
$e=vu \notin M$,
for a vertex
$u$
of
$G$.
Assign color
$2$
to
$e$.
The only vertex that can be a candidate for the edge color code equal to
$v$
is
$u$.
Since
$u$
can not be a pendant vertex, we have
${N}(u)\setminus \{v\} \neq \emptyset$.
If
$|{N}(u)\setminus \{v\}| \geq 2$,
there exists a vertex
$w \in {N}(u)\setminus \{v\}$
such that
$uw \notin M$.
Thus, the color of
$uw$
distinguishes
$v$
and
$u$.
Let
$|{N}(u)\setminus \{v\}| = 1$.
Suppose that
${N}(u)\setminus \{v\}=\{z\}$.
So,
$uz \in M$.
If
$\deg(z) \geq 2$,
then there exists a vertex
$w$
such that
$zw \notin M$.
If there is edge
$vw$,
then we have a cycle with vertices
$v, u, z$,
and
$w$.
Hence, we must have at least one vertex in this cycle with a degree greater than
$2$.
Clearly, vertices
$z$
or
$w$
can have a degree of more than
$2$.
Since
$vw \in M$,
in all possible cases, vertices
$v$
and
$u$
have different an edge color codes. Also, if
$\deg(z)=1$,
there exist an edge
$f \notin M$
with
${d}(v, f) = {d}(u, f)-1$,
and the result is available.

Finally, let
$\deg(v) \geq 3$.
In this case, consider vertices
$x, y$, 
and
$z$
as neighbors of
$v$
such that
$xv \in M$.
This implies that
$yv, zv \notin M$.
Now, the colors of
$yv$
and
$zv$
distinguish 
$v$
with the other vertices of
$G$.
\end{proof}

Theorem \ref{the-match} can be extended for maximum matching.
\begin{theorem}\label{the-max}
Let $G$ be a graph with order $n\geq 5$ and size $m$. If $G$ has a max matching $M$ with $|M|=k$, then
$\chi'_{L}(G) \leq m- k + 1$. This bound is sharp for cycle $C_5$, path $P_5$, star $K_{1,n-1}$ for $n\ge 2$ and double star $S_{p,1}$.
\end{theorem}
\begin{proof}
Let $M$ be a  maximum matching of $G$  with $|M|=k$. It is clear that $M$ saturates $2k$ vertices, and $n-2k$  vertices cannot be saturated by $M$.
We add $n-2k$ vertices to $G$ and make  each of them adjacent to a  vertex that is not saturated. Then, the resulted graph is of order $2n-2k$, size
$m+n-2k$ and has  a perfect matching of size $k+n-2k=n-k$. Now, Theorem \ref{the-match} implies that $\chi'_{L}(G) \leq m+n-2k -(n-k)+ 1=m-k+1$.
\end{proof}

\begin{theorem}
Let $G$ be a graph with order $n \geq 4$ and size $m$. If $G$ has $k$ edge-disjoint perfect matchings $M=M_1 \cup M_2 \cup \dots \cup M_k$ and
$G \backslash M$ is a connected spanning subgraph of $G$, then $\chi_L'(G) \leq m-kn/2 +k$.

\end{theorem}
\begin{proof}
Let $M=M_1 \cup M_2 \cup \dots \cup M_k$ be $k$ edge-disjoint matchings in $G$.  Let $G \backslash M$ be a connected subgraph. Then, establish an edge coloring $\alpha$ on $G$ by assigning a distinct color
to each matching and assigning distinct colors to all remaining edges of $G \backslash M$. Certainly,
this coloring $\alpha$ is an edge proper coloring of $G$. Since  $G \backslash M$ is a connected spanning subgraph, then there are no two vertices incident to the same set of colors. This means that every vertex has a distinct edge color code. Therefore, $\alpha$ is an edge-locating coloring of $G$.
\end{proof}

As a closing remark, we  raise the following question: Is an edge-locating chromatic number of a graph monotonic? Precisely, is it true that if $G$ is a proper subgraph of $H$, then $\chi_L'(G) \leq \chi_L'(H)$?
We know that the metric dimension of a graph is not monotonic, since if $G$ is a star $K_{1,n-1}$ with $n \geq 5$ and $H$ is a graph formed from $G$ by adding one edge connecting two end-point vertices, then $dim(G) > dim(H)$. The  locating chromatic number of a graph is also not monotonic, since if $G=C_4$ and $H$ is a graph formed
from $C_4$ by adding two pendant edges to two consecutive vertices of $C_4$, then $4=\chi_{L}(G) > \chi_{L}(H)=3$.

\begin{theorem}
The edge-locating chromatic number of a graph is not necessarily  monotonic.
\end{theorem}
\begin{proof}
For a positive integer
$n$,
let
$T_n$
denote the perfect binary tree, i.e.,
$T_n$
is a  tree with a root
$r$ of degree $2$ and other vertices of degree $3$ or $1$
in which the distance between the root vertex $r$ and any leaf is
$n$.
Let
$G$
denote the graph obtained from $T_n$ by making adjacent a pendant edge to
$r$ 
in
$T_3$.
We will show that
$\chi_L'(G) \geq 5$.
For a contradiction, since $G$ has at least two vertices of degree $3$, we assume that $\chi_L'(G) = 4$.
Without loss of generality, we may suppose that the incident edges of
$r$
are colored by
$1, 2$,
and
$3$,
such that the leaf is colored by
$1$.
It is clear that $G$ has seven vertices with degree
$3$.
The distance between any vertex of degree
$3$
and any color class is at most $2$.
So, we don't have  more than two vertices of
$G$
with the same colored incident edges. Since $\binom{4}{3}=4$, there are
$4$
ways for coloring the incident edges of a vertex of degree
$3$.
Hence, for exactly one vertex of degree $3$, the edges incident on it take a set of three colors,
and for the rest of the vertices of degree $3$,
for both vertices, the edges incident on them take a set of three colors.
Vertex
$r$
is the only vertex with colored incident edges
$1, 2$,
and
$3$.
Any other vertex with these colored incident edges has distance
$1$
from color
$4$.
We say that the vertices of depth
$i$
in
$G$
are the vertices of
$T_3$
with distance
$i$
from
$r$.
In
$T_3$,
there are two children, as $r_L$ and $r_R$, on the left and right of
$r$.
The children and grandchildren of $r_L$ and $r_R$, called by left part and right part, receptively.
Now, we want to determine the position of two vertices of degree
$3$
with colored incident edges
$2, 3$
and
$4$.
Clearly, these two vertices cannot be in depth $1$ or the same part simultaneously. If these two vertices are in different parts,
then one edge between depth 1 vertex and depth $2$ vertex must be colored by
$1$,
and the other one is not colored by
$1$.
This implies that we have three vertices with colored incident edges
$2, 3$
and
$4$,
which is a contradiction.  Assume that there are two vertices of degree
$3$
with colored incident edges
$2, 3$,
and
$4$
in
depth 1 and depth 2. Similarity, distinguishing these two vertices gives us another vertex with colored incident edges
$2, 3$,
and
$4$,
a contradiction. Therefore, 
$\chi_L'(G) \geq 5$ and obviously, by assigning $5$ colors to the edges of $G$, we can show that $\chi_L'(G)=5$.
Let
$H$
denote the graph obtained from
$G$
with join
$r$
to a pendant vertex in depth $3$. One can check that
$\chi_L'(H)=4$ (see Figure \ref{fig-mon2}).
Therefore, there exist graphs
$G$
and
$H$
that
$G \subset H$
and
$\chi_L'(H) \leq \chi_L'(G)$.
\end{proof}

\begin{figure}[!h]
\begin{center}
\label{FigE2}
  \includegraphics[width=10cm]{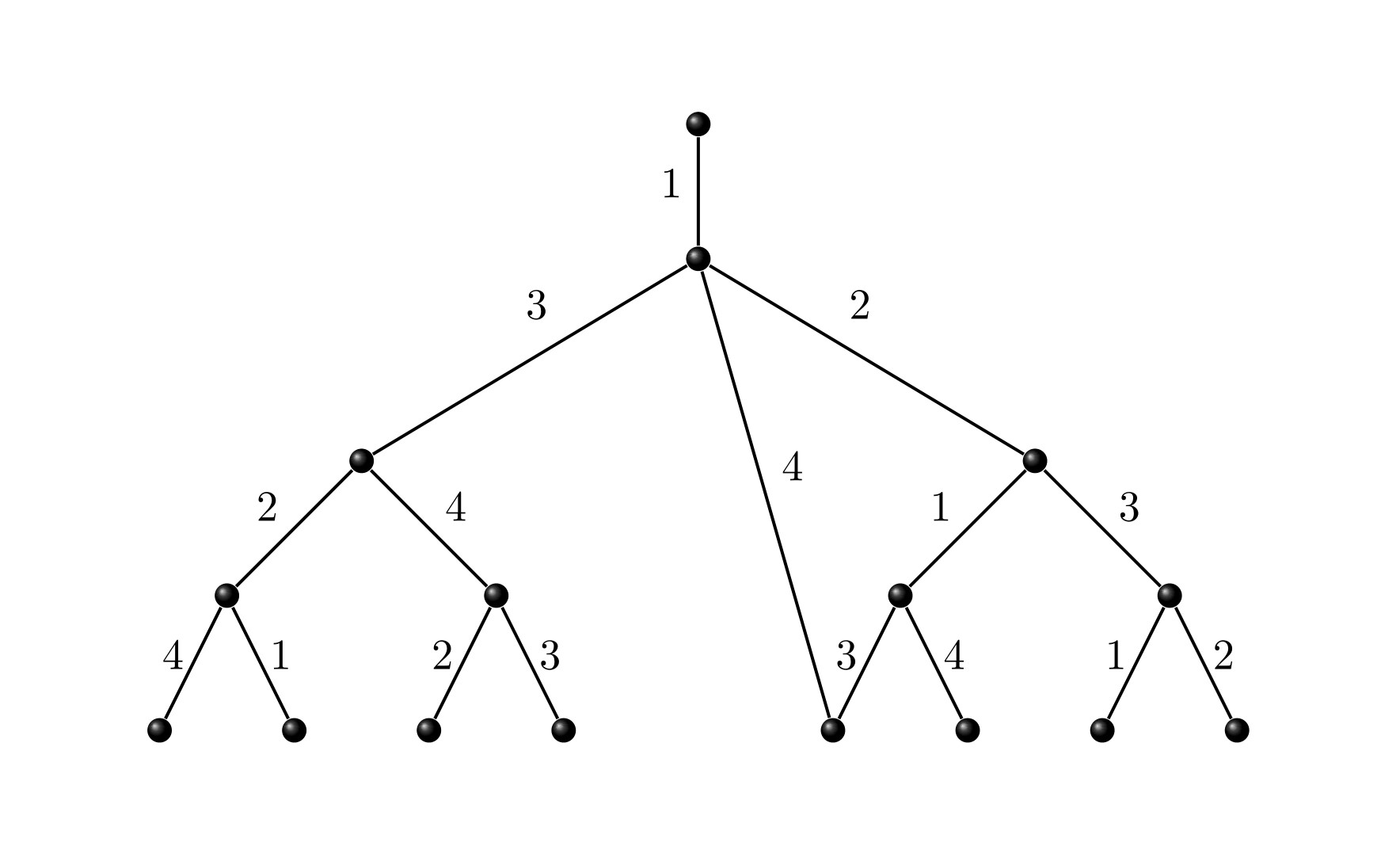}
\end{center}
\caption{\ Graph $H$ with edge-locating chromatic number $4$.}\label{fig-mon2}
\end{figure}

\section{Join graphs}
For any graphs $G$ and $H$, a {\em join graph} between $G$ and $H$, denoted by $G +H$, is a
graph obtained by connecting all vertices of $G$ with all vertices of $H$. In particular, if
$G=K_1$ and $H$ is a cycle $C_n$, the graph $K_1+C_n$ is called a {\em wheel}, and it is denoted by $W_n$.
The graph $K_1+P_n$ is called a {\em fan}, graph and it is denoted by $F_n$.
The graph $K_1+nK_2$ is called a {\em windmill} graph and it is denoted by $Wm(2n)$. The graph $K_2+nK_2$ is called a {\em book} graph with $n$ pages, and it is denoted by $B_{n}$.
In this section,  we will determine the edge-locating chromatic of join graph $G+H$.
\begin{theorem}
For any graphs $G$ and $H$,\\
$\chi_L'(G+H) \geq max\{ \Delta(G)+|V(H)|, |V(G)|+\Delta(H)\}$.
\end{theorem}
\begin{proof}
It is straightforward since $\Delta(G+H)=max\{\Delta(G)+|V(H)|, |V(G)|+\Delta(H)\}$.
\end{proof}
The upper bound is sharp and achieved by a wheel, a fan, or a windmill, as stated in the following theorem.

\begin{theorem} The following are the edge-locating chromatic number for special join graphs:
\begin{itemize}
\item For  $n \geq 4$, $\chi_L'(W_n)=n$.
\item For  $n \geq 4$, $\chi_L'(F_n)=n$; $\chi_L'(F_2)=3$, and $\chi_L'(F_3)=4$.
\item For  $n \geq 3$, $\chi_L'(Wm(n))=2n$, $\chi_L'(Wm(2))=5$
\item For  $n \geq 3$, $\chi_L'(B_n)=2n+2$, and $\chi_L'(B_2)=6$.
\end{itemize}
\end{theorem}

\begin{proof}
For wheels and fans, let $V(W_n)=V(F_n)=\{c,v_1,v_2, \dots, v_n\}$ with a center $c$ and $n \geq 4$.
Since $\Delta(W_n)=\Delta(F_n)=n$ then
$\chi_L'(W_n) \geq n$ and $\chi_L'(F_n) \geq n$.
Now, construct an edge $n$-coloring $\alpha$ of $W_n$ (as well as of $F_n$) as follows.

\[  \alpha(e) = \left\{
\begin{array}{ll}
      i  &  \mbox{, if $e=cv_i$,}   \\
      i+2 \mod n & \mbox{, if $e=v_iv_{i+1}$.} \\
\end{array}
\right. \]
Note that all indices are in mod $n$.
In wheels, the color code of vertex $v_i$ under $\alpha$ will have zero entries in the $i^{th}$, $(i+1)^{th}$, and $(i+2)^{th}$ (in modulo $n$) positions.
In fans, the color code of $v_1$ has zeros in the $1^{st}$ and $3^{rd}$ positions;
the color code of $v_n$ has zeros in the $1^{st}$ and the $n^{th}$ positions. The color codes for other vertices are the same as for wheels. The color code of vertex $c$ has all zero entries. Therefore all color codes are different for wheels as well as on fans. For small cases, it is easy to verify that $\chi_L'(F_2)=3$, and $\chi_L'(F_3)=4$.

In windmills $Wm(n)$,  for $n \geq 3$, let $V(Wm(n))=\{c,v_1,v_2, \dots, v_{2n}\}$ with a center $c$ and
$E(Wm(n))=\{v_iv_{i+1} | \mbox{ for all odd $i \leq 2n$}\} \cup \{cv_i | \mbox{ for all $i \leq 2n$}\}$.
Since $\Delta(Wm(n))=2n$, then $\chi_L'(Wm(n)) \geq 2n$.
Now, construct an edge $(2n)$-coloring $\alpha$ of $Wm(n)$ as follows.

\[  \alpha(e) = \left\{
\begin{array}{ll}
      i  &  \mbox{, if $e=cv_i$,}   \\
      i+2 \mod n & \mbox{, if $e=v_iv_{i+1}$ and $i$ is odd.} \\
\end{array}
\right. \]
Note that all indices are in mod $n$. This coloring $\alpha$ is easily verified as an edge-locating
coloring.

In Books $B_n$, for $n \geq 3$, let $V(B_n)=\{c_1,c_2,v_1, v_2,\cdots, v_{2n-1}, v_{2n} \}$ and
$E(B_n)=\{c_1c_2\} \cup \{c_1v_{2i-1}, c_1v_{2i}, c_2v_{2i-1}, c_2v_{2i} | $ $\mbox{ $ 1\le i \leq n$}\}$
$\cup \{v_{2i-1}v_{2i}| \mbox{  $1\le i \le n$}\}$.
Since $\Delta(B_n)=2n+1$ and there are two vertices of degree $2n+1$, then $\chi_L'(B_n) \geq 2n+2$.
Now, construct an edge $(2n+2)$-locating coloring $\alpha$ of $B_n$ as follows.

\[  \alpha(e) = \left\{
\begin{array}{ll}
      1  &  \mbox{, if $e\in \{c_1c_2, v_{2i-1}v_{2i}|\ 1\le i \le n\}$,}   \\
      i+1 (\mod 2n) & \mbox{, if $e\in \{c_1v_{2i-1}, c_2v_{2i}|\ 1\le i \le n\} $,} \\
      n+2  & \mbox{, if $e=c_1v_{2}$,} \\
      n+2+i (\mod 2n) & \mbox{, if $e\in \{c_1v_{2i+2},c_2v_{2i-1}|\ 1 \le i \le n-1 \}$,} \\
      2n+2 &  \mbox{, if $e=c_2v_{2n-1}$.}\\
\end{array}
\right. \]
It is easy to verify that this coloring $\alpha$ is an edge-locating coloring.
\end{proof}

\begin{theorem}\label{general-join}
Let $G$ be a connected graph and $H=G+K_1$. Then we have

\begin{enumerate}
\item[(i)] If $G$ is graph of order $2n$ and $\Delta(G)\le 2n-2$, then $\chi_L'(H) \leq 2n$. Furthermore, $\chi_L'(H) = 2n+1$ if and only if $G$ has at least one vertex of degree $2n-1$,
\item[(ii)] If $G$ is a graph of order $2n+1$ and $\Delta(G)\le 2n-1$, then $\chi_L'(H) \leq 2n+2$, and equality holds if  $G$ has at least one vertex of degree $2n$
\end{enumerate}

\end{theorem}

\begin{proof}
\indent (i). Let $|V(G)|=2n$. If $\Delta(G)\le 2n-2$, then $G\subseteq K_{2n}\setminus M_n$. From Theorem \ref{KnminusMatchingeven}
$\chi_L'(K_{2n}\setminus M_n) = 2n$ and then $\chi_L'(G) \leq 2n$. Now we have, $H\subseteq K_{2n+1}\setminus M_n$ and from
Theorem \ref{KnminusMatchingodd} $\chi_L'(K_{2n+1}\setminus M_n) = 2n$ and then $\chi_L'(H) \leq 2n$.\\
Now suppose that  $G$ has at least one vertex of degree $2n-1$. Then $H$ has at least two vertices of degree $2n$, and hence $\chi_L'(H) \geq 2n+1$.
On the other hand, $H\subseteq K_{2n+1} \setminus M_{k}$ for $k\le n-1$. By Theorem \ref{KnminusMatchingodd}, $\chi_L'(K_{2n+1}\setminus M_k) = 2n+1$,
and thus  $\chi_L'(H) \leq 2n+1$. Therefore, the equality holds.\\
Conversely, suppose that the equality holds and, in contradiction, $G$ has no vertex of degree $2n-1$, which means that $\Delta(G)\le  2n-2$. From
the first part of the proof, since the order of $G$ is $2n$, hence $\chi_L'(H) \leq 2n$, that is a contradiction.

(ii). Let $|V(G)|=2n+1$. If $\Delta(G)\le 2n-1$, then $G\subseteq K_{2n+1}\setminus M_n\cup M_k$ where $k\ge 1$. From Theorem \ref{KnminusMatchingodd}, 
$\chi_L'(K_{2n+1}\setminus M_n\cup M_k) \le 2n$, and then $\chi_L'(G) \leq 2n$. In this case, $H+K_1$ is a connected graph of order $2n+2$, with exactly one vertex of maximum degree $\Delta(H)=2n+1$.  Thus we have $H\subseteq K_{2n+2}\setminus M_n\cup M_k$ and from
Theorem \ref{KnminusMatchingeven} $\chi_L'(K_{2n+2}\setminus M_n\cup M_k) \le 2n+2$ and then $\chi_L'(H) \leq 2n+2$.\\
Now suppose that  $G$ has at least one vertex of degree $2n$. Then $H$ has at least two vertices of degree $2n+1$ and hence $\chi_L'(H) \geq 2n+2$.
On the other hand, $H\subseteq K_{2n+2} \setminus M_{n}$. By Theorem \ref{KnminusMatchingeven} $\chi_L'(K_{2n+2}\setminus M_n) \le 2n+2$,
and thus  $\chi_L'(H) \leq 2n+2$. Therefore, the equality holds.\\
\end{proof}

\section{Trees}

\begin{theorem}\label{the-dsta}
For any double star $S_{p,q}$,
$\chi'_L(S_{p,q})=\begin{cases}
p+1, & \text{if}\; p>q
\\
p+2, & \text{if}\; p=q.
\end{cases}$.
\end{theorem}

\begin{proof}
Let $G=S_{p,q}$, where $p>q$, with support vertices $v,u$ of degrees $p+1$, $q+1$  and end vertices $v_1,\ldots, v_p, u_1,\ldots, u_q$ respectively.
Then, by K\"{o}nig’s Theorem \cite[Theorem 10.8]{book}, $\chi'(G) = p+1$ and hence
$\chi'_L(S_{p,q})\ge p+1$. On the other hand, if we assign color $i$ to $vv_i$ and $uu_i$, and assign color $p+1$ to the vertex $vu$, then
$c_{\pi}(v_i)=(1,1,\ldots, d(v_i, C_i)=0,1,\ldots,1)$ for $1\le i\le p$, $c_{\pi}(v)=(0,0,\ldots, 0)$,
$c_{\pi}(u)=(0,0,\ldots, 0, d(u, C_{q+1})=1, \ldots, d(u, C_{p})=1, d(u, C_{p+1})=0)$, and $c_{\pi}(u_j)=(1,1,\ldots, d(u_j, C_j)=0, 1,\ldots,d(u_j, C_{q})=1, d(u_j, C_{q+1})=2, \ldots, d(u_j, C_{p})=2, d(u_j, C_{p+1})=1)$ for $1\le j\le q$. Therefore $\chi'_L(S_{p,q})= p+1$.

Let $p=q$. Then $\chi'(S_{p,q}) = p+1$, and edges color $i$ for $vv_i$ and $uu_i$ $1\le i \le p$ and color $p+1$ for $vu$. In this case, 
$c_{\pi}(v)=(0,0,\ldots, 0)=c_{\pi}(u)$. Now by changing the color edge $uu_1$ from $1$ to  $p+2$. Then using above method, it can be seen that
all vertices have distinct edge color codes. Therefore, $\chi'_L(S_{p,q})= p+2$.
\end{proof}

In general we have,
\begin{theorem}\label{5}
 Let $n\ge 4$. There exists a tree $T$ of size $m$ having edge-locating-chromatic
number $k$ if and only if $k \in \{3, 4, \ldots, m-1, m\}$.
\end{theorem}
\begin{proof}
For $k=3$,  consider $T=P_{m+1}$ by Theorem \ref{prop 1}. For $k\ge 4$, let $T$ be a tree with vertex set $\{v_1, v_2,\ldots, v_{m+1}\}$ where vertex $v_2$ is of degree $k$, vertices $v_1, v_3, v_4, \ldots, v_{k}, v_{m+1}$ of degree $1$, and other vertices are of degree $2$. Now if we assign $i$ to edge $v_2v_i$, ($1\le i\ne 2 \le k+1$), assign  $2$ and $1$ to other edges alternately, then for this $T$, it is obvious to see that $\chi'_L(T)=k$.
\end{proof}

\begin{theorem}
Let $T$ be a tree with $k$ support vertices $v_1,v_2,\cdots, v_k$, and $\ell_i$  leaves adjacent to $v_i$, where $\ell_1\le \ell_2 \le \cdots \le \ell_k$. Let $e_{i,j}$ be the pendant edges corresponding to the support vertex $v_i$ where $1\le j\le \ell_i$ and $T'$
be the induced subgraph of non-leaves of $T$.
If $\Delta(T) <m =\sum_{i=1}^{k}\ell_i$. Then $\chi'_{L}(T)\leq \Delta(T')+\ell_k+k-1$. Equality holds if and only if $T=S_{p,p}$.
\end{theorem}

\begin{proof}
We can consider an edge proper $\Delta(T')$-coloring of $T'$, with colors $1, 2, \ldots, \Delta(T')$. Also, color the $\ell_k$ pendant edges with distinct colors $\Delta(T')+1, \Delta(T')+2, \ldots, \Delta(T')+\ell_k$.
Now assign colors $\Delta(T')+1, \Delta(T')+2, \ldots, \Delta(T')+\ell_{i}-1,
\Delta(T')+\ell_{k}+i$ to the edges $e_{i,1}, \cdots,  e_{i,\ell_i}$ if $\ell_i \geq 2$ or assign
color $\Delta(T')+\ell_k+i$ to edge $e_{i,\ell_i}$ if $\ell_i=1$.
Now, let $v$ and $u$ be two arbitrary vertices of $T$. Let $P_{v-u}$ denote the path between $v$ and $u$, and $P$ be a maximal path that contains
$P_{v-u}$. There exist two leaves $e$ and $e'$ on $P$ such that the colors of $e$ and $e'$ are distinct and distinguish vertices $v$ and $u$.

For equality, if $T=S_{p,p}$ ($p\ge 2$) Theorem \ref{the-dsta} deduces the result.\\
Conversely, let $\chi'_{L}(T)= \Delta(T')+\ell_k+k-1$ and $T\ne S_{p,p}$. If $T=S_{p,q}$, where $p\ge q+1$, then Theorem \ref{the-dsta} shows that  $\chi'_{L}(T)=p+1\ne 1+p+1=p+2$, a contradiction.
Hence $T$ has at least $k\ge 3$ support vertices, $\Delta(T')\ge 2$ and $T'$ has at least two leaves, say $v_r, v_t$, and one non leaf, say $v_i$. Suppose $v_k$ is not a leaf in $T'$, since $\ell_r\le \ell_k$, then the pendant edges $e_{r,j}$s  corresponding to $v_r$ can be
colored with the colors of the pendant edges $e_{k,j}$s  corresponding to $v_k$. Suppose $v_k$ is a leaf in $T'$, then the pendant edges $e_{i,j}$s
corresponding to $v_i$ can be colored with the colors of the pendant edges $e_{k,j}$s  corresponding to $v_k$. In the two above cases, other pendant
edges corresponding to other support vertices can be colored by the method in the first part of the theorem. Therefore $\chi'_{L}(T)\le \Delta(T')+\ell_k+k-2$. This  contradiction presents $T=S_{p,p}$.
\end{proof}
\begin{theorem}

Let
$T$
be a tree with
$m\geq 3$
leaves.
If
$\Delta(T)= m$
then
$\chi'_{L}(T)=m$.
If
$\Delta(T)< m$
then
$\chi'_{L}(T)\leq \Delta(T')+m$,
where
$T'$
is the induced subgraph of non-pendant vertices of
$T$.
\end{theorem}
\begin{proof}
Assume first that
$\Delta(T)= m$.
Let
${N}(v)=\{w_1, w_2, \ldots, w_m\}$,
for a vertex
$v$
of
$T$.
Let
$v_1 u_1, v_2 u_2, \ldots, v_m u_m$
be the leaves of
$T$
such that
$v_i$'s
are pendant vertices, for
$1\leq i \leq m$.
For any
$i$, $1\leq i \leq m$,
suppose that
$P_i$
is the
$v-v_i$
path that contains vertex
$w_i$, $1\leq i \leq m$.
Since $\Delta(T)= m$,
$V(P_i) \cap V(P_j) = \{v\}$, for any
$i$ and $j$, $1\leq i, j \leq m$.
Consider a coloring of
$T$
in such a way that for any
$i$, $1\leq i \leq m-1$,
the edges of
$P_i$ ($P_m$)
are colored by colors
$i$ ($m$)
and
$i+1$ ($1$),
alternately, such that edges
$vw_i$
are colored by
$i$,
for
$1\leq i \leq m$.
Any non-pendant vertex of
$P_i$ ($P_m$)
has a distance zero from
$\mathcal{C}_i$ ($\mathcal{C}_m$)
and
$\mathcal{C}_{i+1}$ ($\mathcal{C}_1$),
and distance more than zero from other colors.
Hence, each non-pendant vertex of
$T$
is distinguished by other vertices. On the other hand,
$v_i$ ($v_m$)
has a distance zero from one of the color classes
$\mathcal{C}_i$ ($\mathcal{C}_m$)
and
$\mathcal{C}_{i+1}$ ($\mathcal{C}_{1}$),
and distance one from another class, for any
$i$,
$1\leq i \leq m$,
that
$|V(P_i)| \geq 3$.
There exist only some elements of
${N}(v) \setminus V(P_i)$ (${N}(v) \setminus V(P_m)$)
that can have the same coordinates according to the color classes
$\mathcal{C}_i$ ($\mathcal{C}_m$)
and
$\mathcal{C}_{i+1}$ ($\mathcal{C}_{1}$).
Let
$z \in {N}(v) \cap V(P_i)$.
If
$|V(P_i)| \geq 3$,
then degree
$z$
is
$2$
and the result is obtained. If
$z$
is a pendant vertex, then since
$m\geq 3$,
there exists a color class
$\mathcal{C}_{j}$
such that the distance of
$z$
from
$\mathcal{C}_{j}$
is one and the distance of
$v_i$
from
$\mathcal{C}_{j}$
is more than one. Therefore, all vertices of
$T$
have a different edge color code, and the result is available.

For the other implication, by \cite[Theorem 10.8]{book}, we can consider an edge proper $\Delta(T')$-coloring of
$T'$,
with colors
$1, 2, \ldots, \Delta(T')$.
Also, color the leaves by distinct colors
$\Delta(T')+1, \Delta(T')+2, \ldots, \Delta(T')+m$.
Now, let
$v$
and
$u$
be two arbitrary vertices of
$T$.
Let
$P_{v-u}$
denote the path between
$v$
and
$u$
and
$P$
be a maximal path that contains
$P_{v-u}$.
There exists two leaves, 
$e$
and
$e'$
on
$P$.
The colors of
$e$
and
$e'$
distinguish vertices
$v$
and
$u$,
and the proof is completed.
\end{proof}

\section{Edge metric dimension and distinguishing chromatic index}
\label{EdgeMetric-dist chro}
The minimum size of subset $S$ of edges of graph $G$ that for any two edges
$e$ and $e'$, there exists $f\in S$ such that ${d}(e, f)\neq {d}(e', f)$,
is the {\em edge metric dimension} of $G$ and denoted by $\dim_{E}(G)$.
We say that the set $S$ is an {\em edge basis} of $G$.
Actually, the edge metric dimension of a graph $G$ is the standard metric dimension of the line graph $L(G)$. This concept is introduced and studied by Nasir et. al. \cite{nasir}.
Also, Kalinowski and Pil\'{s}niak introduced the distinguishing chromatic index in \cite{Kalinowski}, wherein the edge distinguishing coloring is an edge proper coloring such that the only color preserving automorphism is the trivial automorphism.
The {\em distinguishing chromatic index}
$\chi'_{D}(G)$ of a graph $G$
is the minimum number of colors that admit an edge distinguishing coloring. In this section, we study some relations between edge-locating coloring and those concepts.

For any subset $S$ of edges of $G$, let $G-S$
denote the subgraph of $G$ with vertex set $V(G)$
and edge set $E(G)\setminus S$.
Let $S$ be an edge basis of
$G$. Consider graph $H:=G-S$ and assign colors
$1, 2, \ldots, \chi'(H)$ to edges $H$ according to a proper edge coloring of
$H$. Also give distinct colors $\chi'(H)+1, \chi'(H)+2, \ldots, \chi'(H)+ |S|$
to elements of $S$. Clearly this coloring is an edge-locating coloring of
$G$. Since $\chi'(H)\leq \chi'(G)$, we have the following bound.
\begin{align}
\chi'(G)\leq \chi'_{L}(G)\leq \chi'(G)+ \dim_{E}(G).
\end{align}
Clearly, this bound is sharp. For instance, let
$G=C_{2n}$.

Let
$vu$
and
$wx$
be two edges in graph
$G$
and
$f \in {Aut}(G)$.
We say that
$f(vu)=wx$,
if
$f(v)=w$, 
and
$f(u)=x$.

\begin{theorem}
Any edge-locating coloring of a graph is an edge distinguishing coloring.
\end{theorem}
\begin{proof}
Let
$G$
be a graph with size
$m$
and
$\pi=(\mathcal{C}_1, \mathcal{C}_2,\ldots, \mathcal{C}_n )$
be the color classes admitted by an edge-locating coloring
$c$
of
$G$.
The result is immediate if
$n=m$.
Assume that
$n < m$.
For a contradiction, suppose that
$c$
is not an edge distinguishing coloring of
$G$.
Thus, there exists an automorphism
$f$
of
$G$
that preserves the coloring, and
$f(e_a)=e_b$
for two edges
$e_a$
and
$e_b$
in
$\mathcal{C}_1$.
Let
$e_a=aa'$,
$e_b=bb'$,
$f(a)=b$
and
$f(a')=b'$.
Consider arbitrary color
$i$
($1\leq i \leq n$)
and let
${d}(a,\mathcal{C}_i)= {d}(a,e^{a}_i)$
and
${d}(b,\mathcal{C}_i)= {d}(b,e^{b}_i)$,
for edges
$e^{a}_i$
and
$e^{b}_i$
with color
$i$.
We will have
\begin{align}
{d}(a,e^{a}_i) = {d}(f(a),f(e^{a}_i)) = {d}(b,f(e^{a}_i))
\end{align}
and
\begin{align}
{d}(b,e^{b}_i) = {d}(f^{-1}(b),f^{-1}(e^{b}_i)) = {d}(a,f^{-1}(e^{b}_i)).
\end{align}
Since
${d}(a,\mathcal{C}_i)\leq {d}(a,f^{-1}(e^{b}_i))$
and
${d}(b,\mathcal{C}_i) \leq {d}(b,f(e^{a}_i))$,
(2) and (3) imply that
${d}(a,\mathcal{C}_i)={d}(b,\mathcal{C}_i)$.
This means that
$c_{\pi}(a)=c_{\pi}(b)$,
a contradiction.
\end{proof}
\begin{corollary}\label{co1}
For any graph
$G$,
\begin{itemize}
\item[(i)]
$\chi'_{D}(G) \leq \chi'_{L}(G)$.
\item[(ii)]
$\chi'_{D}(G)\leq \chi'(G)+ \dim_{E}(G)$.
\end{itemize}
\end{corollary}

By Theorem 16 \cite{Kalinowski}, the equality of Corollary \ref{co1} (ii) is achieved if and only if
$G$
is a path graph,
$C_4$
or
$C_6$.
Also, Theorem 16 \cite{Kalinowski} concludes that
$\chi'_{D}(G) = \chi'_{L}(G) = k$
for
$k \in \{\Delta(G), \Delta(G)+1 \}$.

\section{Future Research}
As you have seen in different sections, the edge-locating chromatic number is related to different and well-known graph concepts. One of them is the edge chromatic index. Recall that
$\chi'(G)\leq \chi'_{L}(G)$,
for a connected graph
$G$.
Classifying connected graphs
$G$
such that
$\chi'(G) =\chi'_{L}(G)$
can be valuable.
Also, one can check if the edge chromatic index is independent of the edge-locating chromatic number. For this purpose, looking for a graph where the edge chromatic index is
$m$
and the edge-locating chromatic number is
$n$,
for any integers $m$ and $n$ that
$m \leq n$.
We think such a graph is available.
For
$k\geq 2$,
let
$T_k$
be the perfect binary tree with root
$a$,
such that
$deg(a)=2$,
other non-pendant vertices have degree $3$, and all pendant vertices have distance
$k$
from
$a$.
By K\"{o}nig’s Theorem \cite[Theorem 10.8]{book}, the chromatic index of
$T_k$
is
$3$,
for any
$k\geq 2$.
But as
$k$
increases, the edge-locating chromatic number of
$T_k$
also increases. If we find the edge-locating chromatic number of
$T_k$ 
and let
$G$
be the graph obtained by joining the rote of
$T_k$
to a star graph, the question is answered. We end the paper with the following problems.
\begin{problem}
Prove or disprove that for any connected graph
$G$
of order
$n$,
$\chi_L'(G) \leq \chi_L'(K_n)$.
\end{problem}
\begin{problem}
Characterize the class  $\Psi$  of connected graphs such that
$G \in \Psi$
if and only if
$\chi'_{D}(G) = \chi'_{L}(G) = k$
for
$k \in \{\Delta(G), \Delta(G)+1 \}$.
\end{problem}
\begin{problem}
For a connected graph
$G$,
is there a significant relationship between
$\chi_{L}(G)$
and
$\chi'_{L}(G)$?
\end{problem}

\section*{Acknowledgment}
The research has been supported by the 2023 PPMI research grant, Faculty of Mathematics and Natural Sciences, Institut Teknologi Bandung, Indonesia.

\end{document}